\documentclass[11pt, a4paper, twoside,leqno]{amsart}
\usepackage[centering, totalwidth = 380pt, totalheight = 600pt]{geometry}
\usepackage{amssymb, amsmath, amsthm, listings, xspace}
\usepackage{enumerate, microtype, stmaryrd, url} 
\usepackage[latin1]{inputenc}
\usepackage[arrow, matrix, tips, curve, graph, rotate]{xy}
\usepackage[retainorgcmds]{IEEEtrantools}
 \usepackage{color}
\definecolor{daRemarkgreen}{rgb}{0,0.45,0}
\usepackage[colorlinks,citecolor=daRemarkgreen,linkcolor=daRemarkgreen]{hyperref}
\usepackage{lscape}
\SelectTips{cm}{10}
\newcommand{\cat}[1]{\mathbf{#1}}

\DeclareMathAlphabet      {\mathbf}{OT1}{cmr}{b}{n}

\newcommand{\cd}[2][]{\vcenter{\hbox{\xymatrix#1{#2}}}}

\makeatletter
\def\matrixobject@{%
   \edef \next@{={\DirectionfromtheDirection@ }}%
   \expandafter \toks@ \next@ \plainxy@
   \let\xy@@ix@=\xyq@@toksix@
   \xyFN@ \OBJECT@}
\let\xy@entry@@norm=\entry@@norm
\def\entry@@norm@patched{%
   \let\object@=\matrixobject@
   \xy@entry@@norm }
\AtBeginDocument{\let\entry@@norm\entry@@norm@patched}
\makeatother

\renewcommand{\phi}{\varphi}

\newcommand{\C}{{\mathcal C}}

\newcommand{\xtor}[1]{\cdl[@1]{{} \ar[r]|-{\object@{|}}^{#1} & {}}}

\makeatletter

\newcommand{\setmanuallabel}[1]{\stepcounter{equation}{\edef\@currentlabel{\theequation}\label{#1}}}
\newcommand{\printmanuallabel}[1]{\stepcounter{equation}\text{(\theequation)}}

\def\hookleftarrowfill@{\arrowfill@\leftarrow\relbar{\relbar\joinrel\sigmaok}}
\def\twoheadleftarrowfill@{\arrowfill@\twoheadleftarrow\relbar\relbar}
\def\leftbararrowfill@{\arrowdoublefill@{\leftarrow\mkern-5mu}\relbar\mapstochar\relbar\relbar}
\def\Leftbararrowfill@{\arrowdoublefill@{\Leftarrow\mkern-2mu}\Relbar\Mapstochar\Relbar\Relbar}
\def\leftringarrowfill@{\arrowdoublefill@{\leftarrow\mkern-3mu}\relbar{\mkern-3mu\circ\mkern-2mu}\relbar\relbar}
\def\lefttriarrowfill@{\arrowfill@{\mathrel\triangleleft\mkern0.5mu\joinrel\relbar}\relbar\relbar}
\def\Lefttriarrowfill@{\arrowfill@{\mathrel\triangleleft\mkern1mu\joinrel\Relbar}\Relbar\Relbar}

\def\hookrightarrowfill@{\arrowfill@{\lhook\joinrel\relbar}\relbar\rightarrow}
\def\twoheadrightarrowfill@{\arrowfill@\relbar\relbar\twoheadrightarrow}
\def\rightbararrowfill@{\arrowdoublefill@{\relbar\mkern-0.5mu}\relbar\mapstochar\relbar\rightarrow}
\def\Rightbararrowfill@{\arrowdoublefill@{\Relbar\mkern-2mu}\Relbar\Mapstochar\Relbar\Rightarrow}
\def\rightringarrowfill@{\arrowdoublefill@\relbar\relbar{\mkern-2mu\circ\mkern-3mu}\relbar{\mkern-3mu\rightarrow}}
\def\righttriarrowfill@{\arrowfill@\relbar\relbar{\relbar\joinrel\mkern0.5mu\mathrel\triangleright}}
\def\Righttriarrowfill@{\arrowfill@\Relbar\Relbar{\Relbar\joinrel\mkern1mu\mathrel\triangleright}}

\def\leftrightarrowfill@{\arrowfill@\leftarrow\relbar\rightarrow}
\def\mapstofill@{\arrowfill@{\mapstochar\relbar}\relbar\rightarrow}

\renewcommand*\xleftarrow[2][]{\ext@arrow 20{20}0\leftarrowfill@{#1}{#2}}
\providecommand*\xLeftarrow[2][]{\ext@arrow 60{22}0{\Leftarrowfill@}{#1}{#2}}
\providecommand*\xhookleftarrow[2][]{\ext@arrow 10{20}0\hookleftarrowfill@{#1}{#2}}
\providecommand*\xtwoheadleftarrow[2][]{\ext@arrow 60{20}0\twoheadleftarrowfill@{#1}{#2}}
\providecommand*\xleftbararrow[2][]{\ext@arrow 10{22}0\leftbararrowfill@{#1}{#2}}
\providecommand*\xLeftbararrow[2][]{\ext@arrow 50{24}0\Leftbararrowfill@{#1}{#2}}
\providecommand*\xleftringarrow[2][]{\ext@arrow 10{26}0\leftringarrowfill@{#1}{#2}}
\providecommand*\xlefttriarrow[2][]{\ext@arrow 80{24}0\lefttriarrowfill@{#1}{#2}}
\providecommand*\xLefttriarrow[2][]{\ext@arrow 80{24}0\Lefttriarrowfill@{#1}{#2}}

\renewcommand*\xrightarrow[2][]{\ext@arrow 01{20}0\rightarrowfill@{#1}{#2}}
\providecommand*\xRightarrow[2][]{\ext@arrow 04{22}0{\Rightarrowfill@}{#1}{#2}}
\providecommand*\xhookrightarrow[2][]{\ext@arrow 00{20}0\hookrightarrowfill@{#1}{#2}}
\providecommand*\xtwoheadrightarrow[2][]{\ext@arrow 03{20}0\twoheadrightarrowfill@{#1}{#2}}
\providecommand*\xrightbararrow[2][]{\ext@arrow 01{22}0\rightbararrowfill@{#1}{#2}}
\providecommand*\xRightbararrow[2][]{\ext@arrow 04{24}0\Rightbararrowfill@{#1}{#2}}
\providecommand*\xrightringarrow[2][]{\ext@arrow 01{26}0\rightringarrowfill@{#1}{#2}}
\providecommand*\xrighttriarrow[2][]{\ext@arrow 07{24}0\righttriarrowfill@{#1}{#2}}
\providecommand*\xRighttriarrow[2][]{\ext@arrow 07{24}0\Righttriarrowfill@{#1}{#2}}

\providecommand*\xmapsto[2][]{\ext@arrow 01{20}0\mapstofill@{#1}{#2}}
\providecommand*\xleftrightarrow[2][]{\ext@arrow 10{22}0\leftrightarrowfill@{#1}{#2}}
\providecommand*\xLeftrightarrow[2][]{\ext@arrow 10{27}0{\Leftrightarrowfill@}{#1}{#2}}

\makeatother

\newcommand{\twocong}[2][0.5]{\ar@{}[#2] \save ?(#1)*{\cong}\restore}
\newcommand{\twoeq}[2][0.5]{\ar@{}[#2] \save ?(#1)*{=}\restore}
\newcommand{\rtwocell}[3][0.5]{\ar@{}[#2] \ar@{=>}?(#1)+/l 0.2cm/;?(#1)+/r 0.2cm/^{#3}}
\newcommand{\ltwocell}[3][0.5]{\ar@{}[#2] \ar@{=>}?(#1)+/r 0.2cm/;?(#1)+/l 0.2cm/^{#3}}
\newcommand{\ltwocello}[3][0.5]{\ar@{}[#2] \ar@{=>}?(#1)+/r 0.2cm/;?(#1)+/l 0.2cm/_{#3}}
\newcommand{\dtwocell}[3][0.5]{\ar@{}[#2] \ar@{=>}?(#1)+/u  0.2cm/;?(#1)+/d 0.2cm/^{#3}}
\newcommand{\dltwocell}[3][0.5]{\ar@{}[#2] \ar@{=>}?(#1)+/ur  0.2cm/;?(#1)+/dl 0.2cm/^{#3}}
\newcommand{\drtwocell}[3][0.5]{\ar@{}[#2] \ar@{=>}?(#1)+/ul  0.2cm/;?(#1)+/dr 0.2cm/^{#3}}
\newcommand{\dthreecell}[3][0.5]{\ar@{}[#2] \ar@3{->}?(#1)+/u  0.2cm/;?(#1)+/d 0.2cm/^{#3}}
\newcommand{\utwocell}[3][0.5]{\ar@{}[#2] \ar@{=>}?(#1)+/d 0.2cm/;?(#1)+/u 0.2cm/_{#3}}
\newcommand{\dtwocelltarg}[3][0.5]{\ar@{}#2 \ar@{=>}?(#1)+/u  0.2cm/;?(#1)+/d 0.2cm/^{#3}}
\newcommand{\utwocelltarg}[3][0.5]{\ar@{}#2 \ar@{=>}?(#1)+/d  0.2cm/;?(#1)+/u 0.2cm/_{#3}}

\newdir{(}{{}*!<0em,-.14em>-\cir<.14em>{l^r}}
\newdir{ (}{{}*!/-5pt/\dir{(}}
\newdir{ >}{{}*!/-5pt/\dir{>}}

\theoremstyle{definition}

\numberwithin{equation}{section}
\theoremstyle{plain}
\newtheorem{Theorem}{Theorem}[section]

\newtheorem{Lemma}[Theorem]{Lemma}

\theoremstyle{definition}

\newtheorem{Definition}[Theorem]{Definition}

\newtheorem{Remark}[Theorem]{Remark}

\newcommand{\f}[1]{\mathbb #1}

\newcommand{\Set}{{\cat{Set}}}

\begin{document}
\leftmargini=2em \title[From types, topological spaces etc. to globular weak $\omega$-groupoids]{Note on the construction of globular weak $\omega$-groupoids from types, topological spaces ETC.}
\author{John Bourke}
\address{Department of Mathematics, Macquarie University, NSW 2109, Australia}
\email{john.d.bourke@mq.edu.au}
\address{Department of Mathematics, Masaryk University,  Brno 6000, Czech Republic}
\email{bourkej@math.muni.cz}
\date{\today}

\maketitle
\begin{abstract}
A short introduction to Grothendieck weak $\omega$-groupoids is given.  Our aim is to give evidence that, in certain contexts, this simple language is a convenient one for constructing globular weak $\omega$-groupoids.  To this end, we give a short reworking of van den Berg and Garner's construction of a Batanin weak $\omega$-groupoid from a type using the language of Grothendieck weak $\omega$-groupoids.
\end{abstract}

\section{Introduction}
Around 2009/2010 van den Berg and Garner \cite{Berg2011Types} and Lumsdaine \cite{Lumsdaine2010Weak} independently showed that a type in intensional type theory gives rise to a weak $\omega$-category in the sense of Batanin \cite{Batanin1998Monoidal}.\begin{footnote}{To be precise, both papers employed the mild reformulation of Batanin's definition given by Leinster in \cite{Leinster2002A-survey}.}\end{footnote} In \cite{Berg2011Types} this weak $\omega$-category was shown, moreover, to be a weak $\omega$-groupoid.\\
Shortly after these papers appeared, Georges Maltsiniotis \cite{Maltsiniotis2010Grothendieck} brought to attention, and simplified, a further globular definition of weak $\omega$-groupoid that first appeared at the beginning of Grothendieck's manuscript \emph{Pursuing Stacks} \cite{Grothendieck1983Pursuing}.\par 
At the end of 2015 I read the papers of van den Berg--Garner and Maltsiniotis  around the same time.  Being struck by the low-tech and transparent nature of the Grothendieck definition, I figured that it should be significantly easier to communicate the main results of \cite{Berg2011Types} by substituting Batanin's weak $\omega$-groupoids for Grothendieck's.\par
The goal of this largely expository note is to explain precisely that.  We give a self contained introduction to Grothendieck weak $\omega$-groupoids and in that language give a direct reworking, attempting nothing original, of the central results and proofs of \cite{Berg2011Types}.  For the reader unfamiliar with type theory let us point out that the main construction applies to topological spaces and Kan complexes as well as to types.  Our thesis is that Grothendieck weak $\omega$-groupoids provide a transparent and workable notion of globular weak $\omega$-groupoid, and our economical reworking of the main result of \emph{loc.cit.} is intended as evidence to that effect.\par
On setting down to write the present note I became aware that a closely related connection between Grothendieck weak $\omega$-groupoids and intensional type theory was already made by Brunerie \cite{Brunerie2013Syntactic} in 2013.  He defined an intensional type theory whose models provide a notion of weak $\omega$-groupoid, and has shown that each type naturally gives rise to a weak $\omega$-groupoid of that kind.  It is expected that these type theoretic weak $\omega$-groupoids, after some minor modifications \begin{footnote}{One needed modification concerns the shapes of operations that are allowed: the contractible contexts of \cite{Brunerie2013Syntactic} encode globular sets such as the free span that are not encoded by the tables of dimensions described here.}\end{footnote}, are essentially the same as the Grothendieck weak $\omega$-groupoids described here, although the precise details of this correspondence are not yet written down.\par
Let us now give a brief summary of what follows.  In Section 2 we recall the notion of Grothendieck weak $\omega$-groupoid.  This material is from \cite{Maltsiniotis2010Grothendieck} up to insignificant notational distinctions.  Section 3 closely follows \cite{Berg2011Types} in introducing identity type categories and iterating the path object construction to build globular objects in such categories.  We additionally point out that topological spaces and Kan complexes form identity type categories.  Section 4 introduces endomorphism globular theories whilst Section 5 reinterprets the main result and proof of \cite{Berg2011Types} using Grothendieck weak $\omega$-groupoids.\par
The author thanks Clemens Berger, Guillaume Brunerie, Richard Garner and Mark Weber for useful discussions on this topic.

\section{Globular theories and $\omega$-groupoids}
\subsection{The globe category and $\omega$-graphs}
The \emph{category of globes} $\f G$ is freely generated by the graph
$$\cd{
0 \ar@<-5pt>[r]_{\tau_{1}} \ar@<5pt>[r]^{\sigma_{1}} & 1 \ar@<-5pt>[r]_{\tau_{2}} \ar@<5pt>[r]^{\sigma_{2}} & \ldots \ar@<-5pt>[r]_{\tau_{n-1}} \ar@<5pt>[r]^{\sigma_{n-1}} & n-1 \ar@<-5pt>[r]_{\tau_{n}} \ar@<5pt>[r]^{\sigma_{n}} & n \ldots}$$
subject to the relations $\sigma_{n} \circ \sigma_{n-1} = \tau_{n} \circ \sigma_{n-1}$ and $\tau_{n} \circ \sigma_{n-1} = \sigma_{n} \circ \sigma_{n-1}$.\par
These relations ensure that $\f G(n,m) = \{\sigma_{n,m},\tau_{n,m}\}$ for $n<m$ where $\sigma_{n,m}$ and $\tau_{n,m}$ are obtained by composing sequences of $\sigma_{i}$'s and $\tau_{i}$'s respectively.  We typically abbreviate $\sigma_{n,m}$ and $\tau_{n,m}$ by $\sigma$ and $\tau$ when the context is clear.\par
A functor $A:\f G^{op} \to \C$ is called an $\omega$-\emph{graph} or \emph{globular object} in $\C$ and is specified by objects $A(n)$ together with morphisms
$$\cd{
A(n) \ar@<0.8ex>[r]^-{s_{n}} \ar@<-0.8ex>[r]_-{t_{n}} & A(n-1)}
$$
where we write $s_{n}=A(\tau_{n})$ and $t_{n}=A(\sigma_{n})$.  Similarly we write $s_{n,m}=A(\tau_{n,m})$ and $t_{n,m}=A(\sigma_{n,m})$, or just $s$ and $t$ if the context is clear.

\subsection{Globular sums and globular products}
A \emph{table of dimensions} is a sequence $\overline{n}=(n_{1}, \ldots ,n_{k})$ of natural numbers with $n_{2i-1} > n_{2i} < n_{2i+1}$ and $k \in \{1,3,5, \ldots\}$.  Given $\overline{n}$, a functor $D:\f G \to \C$ determines a diagram
$$
\xy
(12,0)*+{D(n_{2})}="10"; (36,0)*+{D(n_{4})}="30";(60,-5)*+{\ldots}; (84,0)*+{D(n_{k-1})}="40";
(0,-10)*+{D(n_{1})}="01"; (24,-10)*+{D(n_{3})}="21";(48,-10)*+{D(n_{5})}="41";(72,-10)*+{D(n_{k-2})}="31";(96,-10)*+{D(n_{k})}="51";
{\ar_-{D\tau} "10"; "01"};{\ar^-{D\sigma} "10"; "21"};{\ar_-{D\tau} "30"; "21"};{\ar^-{D\sigma} "30"; "41"};{\ar_-{D\tau} "40"; "31"};{\ar^-{D\sigma} "40"; "51"};
\endxy
$$
in $\C$ whose colimit
is called a \emph{globular sum} and denoted by $D(\overline{n})$.  If all such colimits exist then we say that $\C$ admits $D$-globular sums (or just globular sums).  For $Y:\f G \to [\f G^{op},\Set]$ the globular sums $Y(1)$ and $Y(1,0,2,1,2)$ are depicted below.
$$
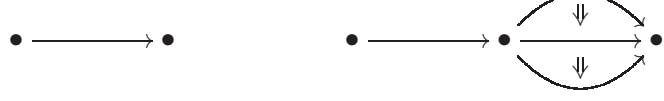
\xy
(0,0)*+{\bullet}="00"; (20,0)*+{\bullet}="10";{\ar^{} "00"; "10"}
\endxy
\hspace{2cm}
\xy
(0,0)*+{\bullet}="00"; (20,0)*+{\bullet}="10";(40,0)*+{\bullet}="20";
{\ar^{} "00"; "10"};{\ar@/^1.5pc/^{} "10"; "20"};{\ar@/_1.5pc/^{} "10"; "20"}; {\ar^{} "10"; "20"};{\ar@{=>}^{}(30,6)*+{};(30,1)*+{}};{\ar@{=>}^{}(30,-1)*+{};(30,-6)*+{}};
\endxy$$
Though we have not labelled them differently, it is important to note that all of the cells depicted above are distinct.\par
Likewise an $\omega$-graph $A:\f G^{op} \to \C$ determines a diagram 
$$
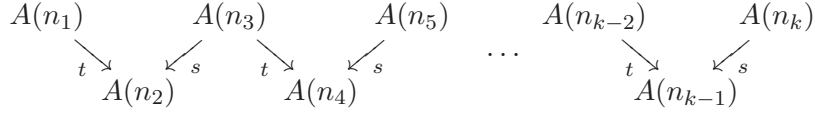
\xy
(12,0)*+{A(n_{2})}="10"; (36,0)*+{A(n_{4})}="30";(60,5)*+{\ldots}; (84,0)*+{A(n_{k-1})}="40";
(0,10)*+{A(n_{1})}="01"; (24,10)*+{A(n_{3})}="21";(48,10)*+{A(n_{5})}="41";(72,10)*+{A(n_{k-2})}="31";(96,10)*+{A(n_{k})}="51";
{\ar_-{t} "01"; "10"};{\ar^-{s} "21"; "10"};{\ar_-{t} "21"; "30"};{\ar^-{s} "41"; "30"};{\ar_-{t} "31"; "40"};{\ar^-{s} "51"; "40"};
\endxy
$$
whose limit, denoted $A(\overline{n})$, is called a \emph{globular product}.

\subsection{Globular theories}
We now describe the category $\Theta_{0}$ that plays the same role for globular theories as the skeletal category of finite sets plays for Lawvere theories.\par To construct $\Theta_{0}$ observe that the category of globular sets $[\f G^{op},\Set]$ is cocomplete and therefore admits $Y$-globular sums.  Taking the full subcategory of $[\f G^{op},\Set]$ on the globular sums yields the initial, up to equivalence, category with globular sums.  $\Theta_{0}$ is a skeleton of this: we can view its objects as the tables of dimensions whilst $\Theta_{0}(\overline{n},\overline{m}) = [\f G^{op},\Set](Y(\overline{n}), Y(\overline{m}))$.  The functor $$D: \f G \to \Theta_{0}$$ factors the Yoneda embedding and is given by $Dn=(n)$ on objects.  We record the universal property of its dual.
\begin{Lemma}
Let $\C$ be a category admitting $A$-globular products.  There exists an essentially unique extension 
$$
\cd{\Theta_{0}^{op} \ar[dr]^{A(-)} \\
\f G^{op} \ar[u]^{D^{op}} \ar[r]_{A} & \C}
$$
of $A$ to a globular product preserving functor $A(-):\Theta_{0}^{op} \to \C$.   This sends $\overline{n}$ to the globular product $A(\overline{n})$.
\end{Lemma}
\begin{Definition}\label{thm:GlobularTheory}
A \emph{globular theory} consists of an identity on objects functor $$J:\Theta_{0}^{op} \to \f T$$ that preserves globular products.
\end{Definition}
The category $Mod(\f T,\C)$ of $\f T$-algebras in $\C$ is the full subcategory of $[\f T,\C]$ containing the globular product preserving functors. Observe that there is a forgetful functor
$$U:Mod(\f T,\C) \to [\f G^{op},\C]$$
given by restriction along $J\circ D^{op}:\f G^{op} \to \f T$.  If $U(X)=A$ then we call $X$ a \emph{$\f T$-algebra structure on $A$.}

\begin{Remark}
The category $\Theta_{0}$ was first described by Berger \cite{Berger2002A-cellular} using level trees.  Globular theories were also first described in \emph{ibid.}, in which the definition was formulated using a sheaf condition equivalent to $J$'s preserving globular products.  The only difference with Definition~\ref{thm:GlobularTheory} is that Definition 1.5 of \emph{ibid.} required that $J$ be faithful, as it typically is.
\end{Remark}

\subsection{Contractibility and weak $\omega$-groupoids}\label{section:Contractibility}
Let $A:\f G^{op} \to \C$.  By a \emph{parallel pair of $n$-cells in $A$} is meant a pair $$f,g:X \rightrightarrows A(n)$$ such that either $n=0$ or $s_{n} \circ f = s_{n} \circ g$ and $t_{n} \circ f = t_{n} \circ g$.  A \emph{lifting} for such a pair is an arrow $h:X \to A(n+1)$ such that 
$$
\cd{& A(n+1) \ar@<-3pt>[d]_-{s} \ar@<3pt>[d]^-{t} \\
X \ar[ur]^-{h} \ar@<3pt>[r]^-{f} \ar@<-3pt>[r]_-{g} & A(n)}
$$
commutes..\par
The $\omega$-graph $A$ is said to be \emph{contractible} if each parallel pair of $n$-cells in $A$ has a lifting, whilst a globular theory $J:\Theta_{0}^{op} \to \f T$ is said to be contractible if its underlying $\omega$-graph $$J \circ D^{op}:\f G^{op} \to \f T$$ is contractible.
\begin{Definition}
A \emph{Grothendieck weak $\omega$-groupoid} is an algebra for some contractible globular theory.
\end{Definition}
Let $J:\Theta_{0}^{op} \to \f T$ be a contractible globular theory and let us agree not to write the action of $J$.  Where are the operations for a weak $\omega$-groupoid in $\f T$?  The map representing composition of 1-cells should have domain the pullback below left.
$$
\cd{(1,0,1) \ar[r]^-{q} \ar[d]_{p} & (1) \ar[d]^-{s} && (1) \ar@<3pt>[d]^-{t} \ar@<-3pt>[d]_-{s} &&& (2) \ar@<3pt>[d]^-{t} \ar@<-3pt>[d]_-{s} \\
(1 )\ar[r]^-{t} & (0) & (1,0,1) \ar@<3pt>[r]^-{s \circ p} \ar@<-3pt>[r]_-{t \circ q}  \ar[ur]^-{m} & (0) & (1,0,1,0,1) \ar@<3pt>[rr]^-{m \circ (m,1)} \ar@<-3pt>[rr]_-{m \circ (1,m)}  \ar[urr]^-{a} && (1) }
$$
Now the parallel $0$-cells in the second diagram admit, by contractibility of $\f T$, a lifting $m$ and this encodes the sought for composition.  Associativity of composition up to a 2-cell is encoded by the lifting $a$ for the parallel $1$-cells in the third diagram.  Weak inverses are encoded by the lifting for the parallel pair 
$$\cd{(1) \ar@<3pt>[r]^-{t} \ar@<-3pt>[r]_-{s} & (0) \hspace{0.5cm} .}
$$
And so on.  For further details see Section 1.7 of \cite{Maltsiniotis2010Grothendieck} or Section 3 of  \cite{Ara2013On-the}.

\begin{Remark}
In \cite{Maltsiniotis2010Grothendieck} a weak $\omega$-groupoid is defined to be an algebra for a \emph{$Gr$-coherator} -- a certain kind of contractible globular theory.  
By Theorem 3.14 of \emph{loc.cit.} the $Gr$-coherators are precisely the \emph{cellular} contractible globular theories and therefore are weakly initial amongst contractible globular theories.  It follows that an $\omega$-graph admits weak $\omega$-groupoid in the present sense just when it admits an algebra structure for a $Gr$-coherator.
\end{Remark}

\section{Identity type categories and iterated path objects}
\subsection{Identity type categories}
An \emph{identity type category} \cite{Berg2011Types} is a category $\C$ equipped with a weak factorisation system $(L,R)$\begin{footnote}{The definition of \cite{Berg2011Types} actually only requires certain factorisations to exist but is equivalent to the present formulation by the argument of Lemma 2.4 of \cite{Shulman2013Univalence}.  See also Lemma 11 of \cite{Gambino2008The-identity} for the original type theoretic argument.}\end{footnote} satisfying the following properties:
\begin{itemize}
\item A terminal object $1$ exists and for each $X \in \C$ the unique map $!:X \to 1$ is an $R$-map.
\item Pullbacks of $R$-maps exist and the pullback of an $L$-map along an $R$-map is again an $L$-map.
\end{itemize}
As shown in \cite{Gambino2008The-identity, Berg2011Types} the syntactic category of an intensional type theory admits the structure of an identity type category.\par
Further examples arise from Quillen model categories $\C$ whose cofibrations are pullback stable along fibrations.  Since weak equivalences between fibrant objects are always stable under pullback along fibrations (see Proposition 13.1.2 of \cite{Hirschhorn2003Model-categories}) the trivial cofibrations between fibrant objects in such model categories are also stable under pullback along fibrations.  So for such $\C$ it follows that the full subcategory of fibrant objects $\C_{f}$ is an identity type category when equipped with the restricted (trivial cofibration/fibration)-weak factorisation system.\par
In the Str\o m model structure on topological spaces \cite{Strom1982The} the cofibrations --\emph{closed cofibrations}-- are stable under pullback along the fibrations, the \emph{Hurewicz fibrations}. This is Theorem 12 of \cite{Strom1969Notes}.  Since all topological spaces are fibrant the category of topological spaces is therefore an identity type category.  In the standard model structure on simplicial sets \cite{Quillen1967Homotopical} the cofibrations are the monos and so are pullback stable along all maps; it follows that the full subcategory of fibrant objects -- \emph{the Kan complexes} -- is an identity type category.
\subsection{Iterating the path object construction}
Starting with an object $X$ of $\C$ the goal now is to build an $\omega$-graph $X_{\star}$ with $X_{\star}(0)=X$. $X_{\star}(1)$ is to be the \emph{path object} of $X$: that is, an $(L,R)$-factorisation
$$
\cd{ X_{\star}(0) \ar[rr]^-{i_{0,1}} && X_{\star}(1) \ar[rr]^-{\langle s_{1},t_{1} \rangle} && X_{\star}(0) \times X_{\star}(0)}
$$
of the diagonal map.  Then $s_{1},t_{1}:X_{\star}(1) \rightrightarrows X_{\star}(0)$ will be the underlying \emph{$1$-graph} of $X_{\star}$.\par
The inductive construction of an $(n+1)$-graph from an $n$-graph makes use of the \emph{$(n+1)$-boundary} $B_{n+1}X_{\star}$ of an $n$-graph.  This has $B_{1}X_{\star} = X_{\star}(0) \times X_{\star}(0)$ whilst for higher $n$, it is given by the pullback below
\begin{equation}\label{eq:induction}
\cd{
B_{n+1}X_{\star} \ar[d]_{q_{n}} \ar[r]^{p_{n}} & X_{\star}(n) \ar[d]^{\langle s_{n}, t_{n}\rangle}\\
X_{\star}(n) \ar[r]_{\langle s_{n}, t_{n}\rangle} & B_{n}X_{\star}}
\end{equation}
in which the map 
\begin{equation}\label{eq:induction2}
\langle s_{n}, t_{n}\rangle: X_{\star}(n) \to B_{n}X_{\star}
\end{equation} is inductively constructed.\par
Let us note that by restriction one can speak of the $(n+1)$-boundary of an $\omega$-graph, and it is not hard to see that this represents parallel pairs of $n$-cells in the $\omega$-graph, as were defined in Section~\ref{section:Contractibility}.\par
Now the pullback \eqref{eq:induction} exists in an identity type category because the inductively defined map \eqref{eq:induction2} is an $R$-map at each stage.  For the inductive step, we observe that the identity on $X_{\star}(n)$ induces a diagonal $\langle 1,1 \rangle:X_{\star}(n) \to B_{n+1}X_{\star}$ whose $(L,R)$-factorisation
$$
\cd{ X_{\star}(n) \ar[rr]^{i_{n,n+1}} && X_{\star}(n+1) \ar[rr]^{\langle s_{n+1},t_{n+1} \rangle} && B_{n+1}X_{\star}}
$$
is taken to define $X_{\star}(n+1)$.  The two maps $s_{n+1},t_{n+1}:X_{\star}(n+1) \rightrightarrows X_{\star}(n)$ then extend $X_{\star}$ to an $(n+1)$-graph.  Because the projections in ~\eqref{eq:induction} are pullbacks of $R$-maps, they are $R$-maps too.  And since $s_{n+1}$ and $t_{n+1}$ are obtained by composing these projections with the $R$-map $\langle s_{n},t_{n} \rangle$ it follows that both $s_{n+1}$ and $t_{n+1}$ are $R$-maps as well.\par
Induction now produces an $\omega$-graph $X_{\star}$ that we call the \emph{iterated path object} and whose relevant properties we now record.
\begin{Lemma}
The iterated path object $X_{\star}$ is a reflexive globular context \cite{Berg2011Types} , i.e.,
\begin{enumerate}
\item There exist $L$-maps $i_{n,n+1}:X_{\star}(n) \to X_{\star}(n+1)$ with $s_{n} \circ i_{n,n+1} = t_{n} \circ i_{n,n+1}$.
\item The maps $s_{n},t_{n}:X_{\star}(n+1) \rightrightarrows X_{\star}(n)$ and $\langle s_{n},t_{n} \rangle:X_{\star}(n+1) \to  B_{n}X_{\star}$ are $R$-maps.
\end{enumerate}
\end{Lemma}

\begin{Remark}
A couple of points are perhaps worth noting.  Firstly, the maps $i_{n,n+1}$ exhibit $X_{\star}$ as a \emph{reflexive} (globular object / $\omega$-graph).  Secondly, the above construction of $X_{\star}$ from $X$ can be understood in terms of the \emph{Reedy structure} on the \emph{reflexive globe category} $\f R$.  For $\f J$ a Reedy category (see \cite{Hovey1999Model} for instance) let $\f J_{\leq n}$ denote the full subcategory on the objects of degree at most $n$.  Then extensions of $A:\f J_{\leq n} \to \C$ to $\f J_{\leq n+1}$ correspond to factorisations of the map
$L_{n}A \to M_{n}A$ from the $n$-th \emph{latching object} of $A$ to the $n$-th \emph{matching object} of $A$, a colimit and limit respectively.   It follows that for $\C$ a sufficiently bicomplete category
equipped with a weak factorisation system, there is a canonical method of inductively constructing an object $X_{\star}:\f J \to \f C$ from $X \in \C$.   Specialised to the Reedy category $\f R$ and an identity type category $(\C,L,R)$ this yields the iterated path object construction.
\end{Remark}

\section{Endomorphism theories}
Let $\C$ be a category with $A$-globular products and consider the extension $A(-):\Theta_{0}^{op} \to \C$ of $A$ as below.
$$\cd{
\Theta_{0}^{op} \ar[drr]|{A(-)} \ar[rr]^{J_{A}} && \textnormal{End}(A) \ar[d]^{K_{A}} \\
\f G^{op} \ar[u]^{D^{op}} \ar[rr]_{A} && \C
}$$
Factoring $A(-)$ as identity on objects followed by fully faithful yields the \emph{endomorphism theory} $$J_{A}:\Theta_{0}^{op} \to \textnormal{End}(A)$$ of $A$.  This has the same objects as $\Theta_{0}$ whilst $\textnormal{End}(A)(\overline{n},\overline{m})=\C(A(\overline{n}),A(\overline{m}))$.  Since $A(-)$ preserves globular products so do both $J_{A}:\Theta_{0}^{op} \to \textnormal{End}(A)$ and $K_{A}:\textnormal{End}(A) \to \C$.  The first fact establishes that \emph{$\textnormal{End}(A)$ is a globular theory} whilst the second exhibits the \emph{canonical $\textnormal{End}(A)$-algebra structure on $A$}.\par
We will use the following lemma, whose proof is a matter of tracing through the definitions, to construct weak $\omega$-groupoids.
\begin{Lemma}\label{thm:endomorphism}
Let $\C$ admit $A$-globular products. Then $\textnormal{End}(A)$ is contractible if and only if each parallel pair $f,g:A(\overline{n}) \rightrightarrows A(m)$ of $m$-cells in $A$ with domain a globular product has a lifting.
\end{Lemma}

\section{The weak $\omega$-groupoid structure}

\begin{Theorem}
Let $\C$ be an identity type category.  The for each $X \in \C$ the iterated path object $X_{\star}$ admits the structure of a weak $\omega$-groupoid.
\end{Theorem}
\begin{proof}
More generally we will show that each reflexive globular context $A:\f G^{op} \to \C$ admits the structure of a weak $\omega$-groupoid.  Firstly we establish some notation.  On composing the $L$-maps $i_{n,n+1}: A(n) \to A(n+1)$ we obtain further $L$-maps $i_{n,m}:A(n) \to A(m)$ for $n < m$ which will be abbreviated by $i$, excepting the case $n=0$ where we write $i_{n}:A(0) \to A(n)$.\par
Now the $L$-maps $i_{n}:A(0) \to A(n)$ assemble into a cone $i:\Delta A(0) \to A \in [\f G^{op},\C]$ under $A(0)$.  This induces a factorisation
$$\cd{
\f G^{op} \ar[dr]_-{A} \ar[r]^-{i/A} & A(0)/ \C \ar[d]^-{U} \\
& \C}
$$
of $A$ through  $A(0)/ \C$.  Here the functor $i/A$ sends $n$ to $i_{n}:A(0) \to A(n)$ whilst $U$ is the forgetful functor.\par
We will prove the theorem by showing:
\begin{enumerate}
\item  The category $A(0)/\C$ has $i/A$-globular products preserved by $U$;
\item The endomorphism theory $\textnormal{End}(i/A)$ is contractible.  
\end{enumerate}
Then the composite 
$$\cd{\textnormal{End}(i/A) \ar[rr]^-{K_{i/A}} && A(0)/\C \ar[rr]^-{U} && \C}$$
will exhibit the structure of a $\textnormal{End}(i/A)$-algebra -- and hence weak $\omega$-groupoid -- on $A$.\par
For (1) we proceed by induction over the length $k$ of a table of dimensions $\overline{n}=(n_{1},\ldots, n_{k})$.  As usual we write $A(\overline{n})$ for the globular product in $\C$ with $p^{\overline{n}}_{j}:A(\overline{n}) \to A(n_{j})$ the $j$'th projection.  We write $i_{\overline{n}}:A(0) \to A(\overline{n})$ for the globular product in $A(0)/ \C$ which then satisfies
\begin{equation*}
\cd{A(0) \ar[r]^{i_{\overline{n}}} & A(\overline{n}) \ar[r]^{p^{\overline{n}}_{j}} & A(n_{j}) & = & A(0) \ar[r]^{i_{n_{j}}} & A(n_{j}) & .}
\end{equation*}
For the base case $\overline{n}=(n_{1})$ we have $A(\overline{n})=A(n_{1})$ with the identity projection, and $i_{n_{1}}:A(0) \to A(n_{1})$ as globular product in $A(0)/\C$.  For $\overline{n}^{+}=(n_{1}, \ldots n_{k},n_{k+1},n_{k+2})$ the globular product $A(\overline{n}^{+})$ in $\C$ can be constructed as the pullback in the rectangle below
\begin{equation}\label{eq:GlobProduct}
\cd{
A(0) \ar@/_1pc/[ddr]_-{i_{\overline{n}}} \ar@/^1pc/[drrr]^-{i_{n_{k+2}}} \ar[dr]^-{i_{\overline{n}^{+}}} \\
& A(\overline{n}^{+}) \ar[d]_-{q} \ar[rr]^-{p^{\overline{n}^{+}}_{k+2}} && A(n_{k+2}) \ar[d]^-{s} \\
& A(\overline{n}) \ar[r]_-{p^{\overline{n}}_{k}} & A(n_{k}) \ar[r]_-{t} & A(n_{k+1})}
\end{equation}
which exists since $s:A(n_{k+2}) \to A(n_{k+1})$ is an $R$-map.  By the universal property of the pullback there exists a unique map $i_{\overline{n}^{+}}:A(0) \to A(\overline{n}^{+})$ rendering commutative the two triangles.  Since $U$ creates pullbacks this is the pullback, and hence globular product, in $A(0)/ \C$.\par
By induction we have now proven (1).  A further consequence of the inductive construction is that the final projection
$$p^{\overline{n}}_{k}:A(\overline{n}) \to A(n_{k})$$ is an $R$-map.  This is trivial in the base case, and clear in the inductive step since the final projection $p^{\overline{n}^{+}}_{k+2}$ is the pullback of a composite $t \circ p^{\overline{n}}_{k}$ of $R$-maps.\par
Now the main ingredient in proving (2) is, in fact, to show that each morphism $$i_{\overline{n}}:A(0) \to A(\overline{n})$$ is an $L$-map and again this is done by induction.  In the base case we have the $L$-map $i_{n_{1}}:A(0) \to A(n_{1})$.  For the inductive step we start by observing that the right vertical arrow $s:A(n_{k+2}) \to A(n_{k+1})$ of \eqref{eq:GlobProduct} has section the $L$-map $i:A(n_{k+1}) \to A(n_{k+2})$.  It follows that its pullback $q$ has a unique section $i^{\prime}$ satisfying the commutativity in the left square below.
$$\cd{
A(\overline{n}) \ar@/^1.5pc/[rr]^{1} \ar[d]_{t \circ p^{\overline{n}}_{k}} \ar[r]^{i^{\prime}} & A(\overline{n}^{+}) \ar[d]_{p^{\overline{n}^{+}}_{k+2}} \ar[r]^{q} & A(\overline{n}) \ar[d]^{t \circ p^{\overline{n}}_{k}} \\
A(n_{k+1}) \ar@/_1.5pc/[rr]_{1} \ar[r]_{i} & A(n_{k+2}) \ar[r]_{s} & A(n_{k+1})
}$$
Since the right and outer rectangles above are pullbacks the left one is a pullback too and, since $p^{\overline{n}^{+}}_{k+2} \in R$ and $i \in L$, it follows that $i^{\prime} \in L$.  Therefore to prove that $i_{\overline{n}^{+}} \in L$ it suffices to show that $i_{\overline{n}^{+}}= i^{\prime} \circ i_{\overline{n}}$.  Both maps give $i_{\overline{n}}$ when postcomposed by $q$.  Composing with the other pullback projection gives
$$p^{\overline{n}^{+}}_{k+2} \circ i^{\prime} \circ i_{\overline{n}} = i \circ s \circ p^{\overline{n}}_{k} \circ i_{\overline{n}} = i \circ t \circ i_{n_{k}} = i \circ i_{n_{k+1}} = i_{n_{k+2}} = p^{\overline{n}^{+}}_{k+2} \circ i_{\overline{n}^{+}}$$
as required.\par
To complete the proof we must show that the endomorphism theory $\textnormal{End}(i/A)$ is contractible.  By Lemma~\ref{thm:endomorphism} this is equally to show that each parallel pair of $m$-cells in $i/A:\f G^{op} \to A(0)/\C$ with domain a globular product $A(\overline{n})$ has a lifting.  Such a parallel pair are depicted below left.
$$\cd{
A(0) \ar[d]_-{i_{\overline{n}}} \ar[dr]^-{i_{m}} && A(0) \ar[d]_-{i_{\overline{n}}} \ar[r]^-{i_{m+1}} & A(m+1) \ar[d]^-{\langle s, t \rangle} \ar@/^1.5pc/@<1ex>[ddr]^-{s} \ar@/^1.5pc/@<-0.5ex>[ddr]_-{t} \\
A(\overline{n}) \ar@<3pt>[r]^-{f} \ar@<-3pt>[r]_-{g} & A(m) & A(\overline{n})\ar@/_1.7pc/@<1ex>[drr]^-{f} \ar@/_1.7pc/@<-0.5ex>[drr]_-{g} \ar[r]_{\langle f, g \rangle} & B_{m+1}A \ar@<3pt>[dr]^-{p_{m}} \ar@<-3pt>[dr]_-{q_{m}} \\
&&&& A(m)}
$$
These induce a unique map $\langle f, g \rangle:A(\overline{n}) \to B_{m+1}A$ to the boundary such that $p_{n} \circ \langle f, g \rangle = f$ and $q_{n} \circ \langle f, g \rangle = g$.  In the diagram above right all paths from $A(0)$ to $A(m)$ coincide as $i_{m}:A(0) \to A(m)$.  Since the pullback projections $p_{m}$ and $q_{m}$ are jointly monic it follows that the square commutes.  Now $i_{\overline{n}}$ is an $L$-map and $\langle s, t \rangle$ an $R$-map.  Therefore there exists a diagonal filler $h:A(\overline{n}) \to A(m+1)$ in the square and this gives the desired lifting.
\end{proof}

\begin{Remark}
The preceding construction of a Grothendieck weak $\omega$-groupoid is simpler than that of a Batanin weak $\omega$-groupoid for a couple of reasons.  One is that Batanin's weak $\omega$-groupoids are defined as a special case of his weak $\omega$-categories and so another step is required.  Another reason is that endomorphism theories seem easier to handle with globular theories rather than globular operads.
\end{Remark}

\end{document}